\documentclass[10pt]{amsart}

\usepackage{amsfonts,amssymb,amsmath,mathrsfs}
\usepackage{amsthm}
\usepackage{graphicx,epic,eepic}
\usepackage{epsfig}
\usepackage{latexsym}
\usepackage{color}
\usepackage{esint}

\usepackage{amsrefs}

\newtheorem{theorem}{Theorem}[section]
\newtheorem{lemma}[theorem]{Lemma}
\newtheorem{proposition}[theorem]{Proposition}
\newtheorem{corollary}[theorem]{Corollary}
\theoremstyle{definition}
\newtheorem{definition}[theorem]{Definition}

\theoremstyle{remark}
\newtheorem{remark}[theorem]{Remark}

\numberwithin{equation}{section}

\def\neweq#1{\begin{equation}\label{#1}}
\def\endeq{\end{equation}}
\def\eq#1{(\ref{#1})}

\def\super{\overline}

\newcommand{\R}{{\mathbb R}}

\newcommand{\eps}{\varepsilon}
\renewcommand{\a}{\alpha}
\renewcommand{\b}{\beta}

\newcommand{\var}{\varphi}

\definecolor{violet}{rgb}{0.5,0,0.8}

\newcommand{\ov}{\overline v}
\newcommand{\ou}{\overline u}

\newenvironment{proofp}[1]{\par\medskip\noindent{\it Proof{ #1.}}}

\begin{document}

\title[The Gel'fand problem for the biharmonic operator]{The Gel'fand problem for the biharmonic operator}

\author{}
\address{}
\email{}

\author[L. Dupaigne]{Louis Dupaigne}

\address{LAMFA, UMR CNRS 7352, Universit\'e de Picardie Jules Verne, 33 rue St Leu, 80039, Amiens Cedex, France}

\email{louis.dupaigne@math.cnrs.fr}

\author[M. Ghergu]{Marius Ghergu}

\address{School of Mathematical Sciences, University College Dublin, Belfield, Dublin 4, Ireland}

\email{marius.ghergu@ucd.ie}

\author[O. Goubet]{Olivier Goubet}

\address{LAMFA, UMR CNRS 7352, Universit\'e de Picardie Jules Verne, 33 rue St Leu, 80039, Amiens Cedex, France}

\email{olivier.goubet@u-picardie.fr}

\author[G. Warnault]{Guillaume Warnault}

\address{LMA, UMR CNRS 5142, Universit\'e de Pau et Pays de l'Adour, Avenue de l'Universit\'e, BP 1155, 64013 Pau Cedex, France}

\email{guillaume.warnault@univ-pau.fr}

\subjclass[2010]{XXXX}






\maketitle

%
{

\begin{abstract}We study stable solutions of the equation $\Delta^2 u = e^{u}$, both in entire space and in bounded domains.
\end{abstract}

\section{Introduction}
A classical result attributed\footnote{see H. Brezis, \cite{brezis}.} to G.I. Barenblatt asserts that there exists infinitely many solutions to the equation
\begin{equation} \label{gelfand}
-\Delta u =2 e^{u}\qquad\text{in $\Omega$},
\end{equation}
whenever  $\Omega$ is the unit ball of $\R^3$ and the equation is supplemented with a homogeneous Dirichlet boundary condition. The result appeared in a volume edited by I.M. Gel'fand \cite{gelfand}, whose name the problem now bears. We refer the interested reader to the book \cite{book} for some of the developments of this equation in the more than sixty years that separate us from Barenblatt's discovery. Let us simply mention that K. Nagasaki and T. Suzuki \cite{ns} completly classified the solutions found by Barenblatt according to their Morse index. Much of what can be said of the equation posed in a general domain $\Omega$ rests, through a blow-up analysis, upon Liouville-type theorems for finite Morse index solutions of the equation posed in entire space. This lead N. Dancer and A. Farina \cite{df} to proving that in dimension $3\le N\le 9$, any solution to \eqref{gelfand} in $\Omega=\R^N$ is unstable outside every compact set and so, it has infinite Morse index.

\noindent In the present work, we consider the fourth-order analogue of the Gel'fand problem. Motivated by the aforementioned results and by the expanding literature on fourth-order equations, see in particular the books \cite{ggs} by F. Gazzola, H.-Ch. Grunau, and G. Sweers, and \cite{egg} by P. Esposito, N. Ghoussoub, and Y. Guo, we want to classify solutions of
\begin{equation} \label{gelfandbe}
\Delta^2 u = e^{u}\qquad\text{in $\R^N$},
\end{equation}
which are stable (resp. stable outside a compact set), that is, solutions such that
\begin{equation} \label{stability}
\int_{\Omega}e^{u}\varphi^2\;dx \le \int_{\Omega}\vert\Delta\varphi\vert^2\;dx\quad\text{for all $\varphi\in H^2(\Omega)\cap H^1_{0}(\Omega)$,}
\end{equation}
where $\Omega$ is $\R^N$(resp. the complement of some compact subset of $\R^N$).  
Consider first radial solutions.  Noting that the equation is invariant under the  scaling transformation
\begin{equation} \label{scaling} u_{\lambda}(x) = u(\lambda x) + 4\ln\lambda, \quad x\in\R^N, \lambda>0,\end{equation}
we may always assume that $u(0)=0$.
\begin{proposition} \label{prop radial} Let $5\le N\le12$. Assume $u$ is a radial solution of \eqref{gelfandbe}. Let $v=-\Delta u$, $0=u(0)$, and $\beta=v(0)$. There exist $\beta_{1}>\beta_{0}>0$ depending on $N$ only such that
\begin{itemize}
\item $\beta\ge\beta_{0}.$
\item If $\beta=\beta_{0}$, $u$ is unstable outside every compact set.
\item If $\beta\in(\beta_{0},\beta_{1})$, $u$ is unstable, but $u$ is stable outside a compact set.
\item If $\beta\ge\beta_{1}$, $u$ is stable.
\end{itemize}
\end{proposition}
\begin{remark} The fact that no radial solution exists for $\beta<\beta_{0}$ is due to G. Arioli, F. Gazzola and H.-C. Grunau \cite{agg}. In addition, E. Berchio, A. Farina, A. Ferrero, and F. Gazzola first proved in \cite{bffg} that for $5\le N\le 12$, $u$ is stable outside a compact set\footnote{the notion of stability outside a compact set that we use here is stronger than the one given in \cite{bffg}. One can easily check that the results of \cite{bffg} remain true in our setting.} if and only if $\beta>\beta_{0}$. The novelty here is the number $\beta_{1}$: our result characterizes stable radial solutions when $5\le N\le12$. See \cite{bffg}, \cite{warnault} for the remaining cases $1\le N\le 4$ and $N\ge13$.
\end{remark}
In particular, there is no hope of proving a result similar to that of Dancer and Farina in our context, without further restrictions on the solution. One might ask whether all stable solutions are radial, at least in dimension $5\le N\le12$. This is still not the case.
\begin{theorem}\label{nonradialstable}
Assume $N\geq 5$. Take a point $x^0=(x^0_{1},\dots,x^0_{N})\in\R^N$, 
parameters $\alpha_{1},\dots,\alpha_{N}>1+N/2$, and let
\begin{equation}\label{polynomial}  p(x)=\sum_{i=1}^N\alpha_{i}(x_i-x^0_i)^2.\end{equation} Then, there exists a solution $u$ of \eq{gelfandbe} such that
\neweq{asympt1}
u(x)=-p(x)+{\mathcal O}(|x|^{4-N})\quad\mbox{ as }|x|\rightarrow
\infty,
\endeq
In particular, $u$ is stable outside a compact set (resp. stable, if $\min_{i=1,\dots,N}\alpha_{i}$ is large enough) and $u$ is not radial about any point if the coefficients $\alpha_{i}$ are not all equal.
\end{theorem}
\begin{remark}
Using the scaling \eqref{scaling}, one immediately obtains a solution $u$ of \eq{gelfandbe} such that
$u(x)=-p(x)+C + {\mathcal O}(|x|^{4-N})$ as $|x|\rightarrow \infty$, under the sole assumption that $\alpha_i>0$ for all $i=1,\dots, N$.
\end{remark}
\begin{remark}In fact, any solution satisfying \eqref{asympt1} has finite Morse index, thanks to the Cwikel-Lieb-Rozenbljum formula (see G.V. Rozenbljum \cite{r} and D. Levin \cite{l} for the formula, as well as A. Farina \cite{farina} for its application to semilinear elliptic equations). We thank A. Farina for this observation.
\end{remark}

\begin{remark}In dimension $N\ge13$, the radial solution $u_{\beta0}(x)=-4\ln\vert x\vert + \mathcal O(1)$ is stable, see \cite{bffg}. Observe that $u_{\beta_{0}}$ does not satisfy \eqref{asympt1}. In dimension $5\le N\le 12$, it would be interesting to determine whether, up to rescaling and rotation, all stable solutions do satisfy \eqref{asympt1}.
\end{remark}
All stable solutions that we have encountered so far have quadratic behavior at infinity. In particular, letting
$$v=-\Delta u\quad\text{ and }\quad\super v(r)=\fint_{\partial B_{r}}v\;d\sigma,$$ these solutions satisfy $\super v(\infty)>0$, where \footnote{For any solution of \eqref{gelfandbe}, $v$ is superharmonic. In particular, its spherical average $\super v$ is a decreasing function of $r$. 
We will prove that $v>0$, so that $\super v(\infty)\ge0$ is always well-defined.}
$$ \super v(\infty):=\lim_{r\to+\infty}\super v(r).$$
This motivates the following Liouville-type result.
\begin{theorem}\label{liouville} Assume $5\le N\le 12$. Let $u$ be a solution of \eqref{gelfandbe} such that $\super v(\infty)=0$. Then, $u$ is unstable outside every compact set.
\end{theorem}
\begin{remark}As observed in \cite{bffg}, in dimension $N=4$, applying inequality \eqref{stability} with a standard cut-off function, we easily see that if $u$ is stable outside a compact set, then $e^{u}\in L^1(\R^4)$. Thanks to the work of C.-S. Lin \cite{lin}, up to a rotation of space, $u$ must satisfy
\neweq{asympt2}
u(x)=-p(x)-4\gamma\ln\vert x\vert+c_{0}+{\mathcal O}(|x|^{-\tau})\quad\mbox{ as }|x|\rightarrow
\infty,
\endeq
where $p(x)$ is of the form \eqref{polynomial} with $\alpha_{i}\ge0$, $\gamma=\frac1{32\pi^2}\int_{\R^4}e^{u}\;dx\le 2$, $c_{0}$ and $\tau>0$ are constants. Conversely, there exist solutions of the form \eqref{asympt2}, as proved by J. Wei and D. Ye in \cite{wei-ye}. If in addition $\super v(\infty)=0$, then $\gamma =2$ and up to translation and the scaling \eqref{scaling},
$$
u(x)=-4\ln\left(1+\vert x\vert^2\right)+\ln 384,\qquad\text{for all $x\in\R^4$.}
$$
\end{remark}
Now, let us turn to bounded domains. We begin by recalling a few known results on the Gel'fand problem for the biharmonic operator, when the domain is the unit ball and the equation is supplemented with a homogeneous Dirichlet boundary condition, i.e. for $\lambda>0$, we consider the equation
$$
\left\{
\begin{aligned}
&\Delta^2 u= \lambda e^u&&\quad\mbox{ in }B,\\
&u=\vert \nabla u\vert=0&&\quad\mbox{ on }\partial B.
\end{aligned}
\right.
$$
It is known that there exists an extremal parameter $\lambda^*=\lambda^*(N)>0$ such that the problem has at least one solution (which is stable) for $\lambda<\lambda^*$, a weak stable solution $u^*$ for $\lambda=\lambda^*$ and no solution for $\lambda>\lambda^*$ (see G. Arioli, F. Gazzola, F.-C. Grunau and E. Mitidieri \cite{aggm}). The unique weak stable solution $u^*$ associated to $\lambda=\lambda^*$ is classical if and only if $1\le N\le 12$ (see the work by J. D\'avila, I. Guerra, M. Montenegro and one of the authors \cite{ddgm}, as well as a simplification due to A. Moradifam \cite{moradifam}). It is also known that if $N\ge 5$, the problem has a unique singular radial solution for some $\lambda=\lambda_{S}$ (as follows from the analysis in \cite{aggm} and \cite{ddgm}) and, as in Barenblatt's result, that it has infinitely many regular radial solutions for the same value of the parameter (see the delicate work due to J. D\'avila, I. Flores and I. Guerra \cite{dfg}). The case of general domains with Dirichlet boundary conditions is essentially unexplored, due to the lack of a comparison principle. The case of homogeneous Navier boundary conditions, namely the equation
\neweq{bi}
\left\{
\begin{aligned}
&\Delta^2 u= \lambda e^u&&\quad\mbox{ in }\Omega,\\
&u=\Delta u=0&&\quad\mbox{ on }\partial\Omega,
\end{aligned}
\right.
\endeq
where $\Omega$ is a smooth and bounded domain in $\R^N$, $N\geq 1$, seems, for now, easier to deal with.
Our results continue a story initiated by C. Cowan, P. Esposito, and N. Ghoussoub in \cite{ceg}. They can be summarized as follows.
\begin{theorem}\label{th reg}
Let $N\ge1$ and let $\Omega$ be a smoothly bounded domain of $\R^N$. Let $u^*$ be the extremal solution of \eq{bi}.
\begin{itemize}
\item If $1\le N\le 12$, then $u^*\in C^\infty(\super\Omega)$.
\item If $N\ge 13$,  then $u\in
C^\infty(\Omega\setminus\Sigma)$, where $\Sigma$ is a closed set whose
Hausdorff dimension is bounded above by
$$
{\mathcal H}_{dim}(\Sigma)\leq N-4p^*
$$
and $p^*>3$ is the largest root of the polynomial $(X-{\frac12})^3-8(X-{\frac12})+4$.
\end{itemize}
\end{theorem}
\begin{remark}Theorem \ref{th reg} was first proved for $1\le N\le 8$ by C. Cowan, P. Esposito, and N. Ghoussoub, see \cite{ceg}. As we were completing this work, we learnt that C. Cowan and N. Ghoussoub just improved their result to $1\le N\le10$, see \cite{cg}. The question of partial regularity in large dimension was studied by K. Wang for the classical Gel'fand problem (see \cite{wang}). Unfortunately, we could not understand a part of his proof\footnote{the relation between the exponents in his interpolation inequality (3.18) seems incorrect.}, and our approach is somewhat different. We expect that the computed exponent $p^*$ is not optimal. Similarly, as observed by P. Esposito \cite{esposito}, the methods that we develop here will not yield the expected critical curve (resp. dimension) for the Lane-Emden system (resp. for the MEMS problem), for which new ideas must be found. It could be interesting to see if they improve the results obtained in the previously mentioned references \cite{ceg}, \cite{cg}, as well as the works of J. Wei and D. Ye \cite{wy2}, of J. Wei, X. Xu and W. Yang \cite{wxx}, and of C. Cowan \cite{cowanjuly}.
\end{remark}
Finally, our Liouville-type result, Theorem \ref{liouville}, will be used to prove the following.
\begin{theorem}\label{th reg2}
Let $5\le N\le12$ and let $\Omega$ be a smoothly bounded domain of $\R^N$. Assume in addition that $\Omega$ is convex. Let $u\in C^4(\super\Omega)$ be any classical solution of \eq{bi} and let $v=-\Delta u$. There exists a compact subdomain $\omega\subset\Omega$ such that if
\neweq{3061bis}
\int_{B_r(x_{0})} v\;dx\leq K r^{N-2},
\endeq
for every ball $B_{r}(x_{0})\subset\omega$ and for some constant $K>0$,
then, there exists a number $C$ depending only on $\lambda$, $\Omega$, $N$, $K$ and the Morse index of $u$, such that
$$
\| u \|_{L^\infty(\Omega)}\le C.
$$
\end{theorem}
\begin{definition}
In the above, the Morse index of a solution $u$ is the number of negative eigenvalues of the linearized operator $L=\Delta^2-\lambda e^{u}$ with domain
$D(L)=\{u\;:\; u,\Delta u\in H^2(\Omega)\cap H^1_{0}(\Omega)\}$. According to standard spectral theory, this number is finite.
\end{definition}
\begin{remark}
If $u$ is stable, then \eqref{3061bis} automatically holds for some constant $K$ depending only on $\Omega, N,$ and $\omega$. See Lemma \ref{30juin2012}. We do not know whether this remains valid for solutions of bounded Morse index.
In addition, it would be interesting to know how the number $C$ depends on the Morse index of $u$. See the work of X.-F. Yang \cite{yang} for a result in this direction, in a subcritical setting.
\end{remark}

\noindent {\bf Notation}. For any given function $f$, $\super f$
connotes the spherical average of $f$. We write $f \lesssim g$ (resp. $f \gtrsim g$), when there exists
a numerical constant $C$ such that $f\leq Cg$ (resp. $f\geq C g$). $B_R(x)$ denotes the ball centered at $x$ and of radius $R$.
For shorthand, $B_R=B_R(0)$ and $A_R$ is the annulus of radii $R$ and $2R$.}

\section{Classification of stable radial solutions}
We prove here Proposition \ref{prop radial}.
Take $\beta\ge \beta_{0}$ and let $u=u_{\beta}$ be the radial solution such that $u(0)=0$ and $v(0)=\beta$. 
%
We claim that if $\b$ is large enough, then $u_\b$ is stable. We shall use the following inequality, found in \cite{agg}:
\begin{equation}\label{es1}
u_{\b}\leq -\frac{\beta -\b_0}{2N}r^2\quad \text{for all } r>0.
\end{equation}
Simply choose $\overline \b>\b_0$ such that
\begin{equation}\label{es2}
e^{-\frac{\overline\b-\b_{0}}{2N}r^2}\leq \frac{N^2(N-4)^2}{16r^4} \quad \text{for all } r>0.
\end{equation}
Combining \eqref{es1}, \eqref{es2}, and  the Hardy-Rellich inequality, we deduce that $u_\b$ is stable for $\beta\ge\overline\beta$. So, we may define $\Lambda=\{\b>\b_0 \ |\ u_\b \mbox{ is stable}\}$ and $\b_1= \inf \Lambda$. By standard ODE theory,  one easily proves that $\b_1= \min \Lambda$. According to \cite{bffg}, $u_{\beta_{0}}$ is unstable. So, $\beta_{1}>\beta_{0}$.
Also, by a result in \cite{agg}, solutions are ordered : if $\tilde \b >\beta$, then $u_{\tilde \b}\leq u_\b$. So, $\Lambda$ is the interval $[\beta_{1},+\infty)$. \hfill\qed
\section{Construction of nonradial solutions}
We present here the proof of Theorem \ref{nonradialstable}.
Take a polynomial $p$ of the form \eqref{polynomial}. Without loss if generality, we may assume that {$x^{0}=0$}. 
 We look for a solution $u$ of  the form $u=-p(x)+z(x)$, so that $z$ and $w=-\Delta z$ satisfy
\neweq{nonrad}
\left\{
\begin{aligned}
-&\Delta z=w&&\quad\mbox{ in }\R^N,\\
-&\Delta w=e^{-p(x)}e^z&&\quad\mbox{ in }\R^N.
\end{aligned}
\right.
\endeq

For $x\in \R^N$, let
$$
Z(x)=(1+|x|^2)^{2-N/2}\,,\quad W(x)=(1+|x|^2)^{1-N/2}.
$$
We claim that $(Z,W)$ is a super-solution of \eq{nonrad}. Indeed, straightforward calculations yield
$$
\begin{aligned}
-\Delta Z&\geq 2(N-4)W&&\quad\mbox{ in }\R^N,\\
-\Delta W&=N(N-2)(1+|x|^2)^{-1-N/2}&&\quad\mbox{ in }\R^N.
\end{aligned}
$$
Since $Z \leq 1$ and since we assumed that $p(x)\geq (1+N/2)|x|^2$ in $\R^N$, we have
$$
-\Delta W\geq (1+|x|^2)^{-1-N/2} e^{Z}\geq e^{-(1+N/2)|x|^2}e^{Z}\geq  e^{-p(x)}e^{Z}\quad\mbox{ in }\R^N,
$$
which proves our claim.

Since the system is cooperative, and $(0,0)$ and $(Z,W)$ form a pair of ordered sub- and super-solutions, we obtain the existence of a solution of
\eq{nonrad} which further satisfies $0<z\leq Z$ and $0<w\leq W$ in $\R^N$. Hence, $u(x):=-p(x)+z(x)$ is a solution of \eq{gelfandbe}.

To prove that  $u$ is stable outside a compact set, let us observe again that $Z\leq 1$ in $\R^N$. So, we can find $\rho>0$ large such that
\neweq{hardy1}
e^{u(x)}\leq e^{Z(x)}e^{-p(x)}\leq e^{1-p(x)}\leq  \frac{N^2(N-4)^2}{16|x|^4},\quad\text{for all $|x|> \rho$.}
\endeq
By the Hardy-Rellich inequality, $u$ is stable outside $\overline{B_{\rho}}$.
Remark now that if $\min_{i=1\dots N}\alpha_i>0$ is large enough then \eq{hardy1} is valid for all $x\in \R^N$, so $u$ is stable in $\R^N$. This ends the proof of Theorem \ref{nonradialstable}.
\hfill\qed

\section{Regularity of the extremal solution in dimension $1\le N\le 12$}
In this section, we prove the first part of Theorem \ref{th reg}.
Let $u$ denote the minimal solution to \eqref{bi} associated to a parameter $\lambda\in(\lambda^*/2,\lambda^*)$.
Up to rescaling, we may assume that $\lambda=1$.
The first ingredient in our proof is the following consequence of the stability inequality \eqref{stability}: 
\neweq{weaks}
 \int_\Omega e^{\frac u2} \varphi^2\;dx\le \int_\Omega |\nabla \varphi|^2\;dx  \quad \mbox{for all } \varphi\in  H^1_{0}(\Omega).
\endeq
\eqref{weaks} was proved independently\footnote{The result was first made available online by C. Cowan and N. Ghoussoub. See also the Appendix for another result found by several authors.} by C. Cowan and N. Ghoussoub in \cite{cg} and A. Farina, B. Sirakov, and one of the authors in \cite{dfs}. See e.g. Lemma \ref{lem interp} below for its proof.
Now, Let us write the problem \eq{bi} as a system in the following way:
\neweq{systembi}
\left\{
\begin{aligned}
-&\Delta u=v&&\quad\mbox{ in }\Omega,\\
-&\Delta v=e^u&&\quad\mbox{ in }\Omega,\\
&u=v=0&&\quad\mbox{ on }\partial\Omega.
\end{aligned}
\right.
\endeq
Fix $\alpha>{\frac12}$ and multiply the first equation in \eq{systembi} by $e^{\alpha u}-1$. Integrating over $\Omega$, we obtain
$$
\int_{\Omega} \left( e^{\alpha u}-1\right)v\;dx = \alpha \int_{\Omega}e^{\alpha u}\vert\nabla u\vert^2\;dx=
\frac4\alpha \int_{\Omega}\left\vert\nabla\left(e^{\frac{\alpha u}2}-1\right)\right\vert^2\;dx.
$$
By \eqref{weaks},
$$
\int_{\Omega}e^{{\frac{u}{2}}}\left(e^{\frac{\alpha u}{2}}-1\right)^2\;dx\le \int_{\Omega}\left\vert\nabla\left(e^{\frac{\alpha u}2}-1\right)\right\vert^2\;dx.
$$
Combining these two inequalities, we deduce that
\begin{equation} \label{2161}
\int_{\Omega}e^{{\frac{u}{2}}}\left(e^{\frac{\alpha u}{2}}-1\right)^2\;dx \le \frac{\alpha}{4}\int_{\Omega}\left(e^{\alpha u}-1\right)v\;dx.
\end{equation}
Similarly, multiply the second equation in \eqref{systembi} by $v^{2\alpha-1}$ and use \eqref{weaks} to deduce that
\begin{equation} \label{2162}
\int_{\Omega}e^{{\frac{u}{2}}}v^{2\alpha}\;dx\le \frac{\alpha^2}{2\alpha-1}\int_{\Omega}e^{u}v^{2\alpha-1}\;dx.
\end{equation}
By H\"older's inequality,
\begin{align*}
\int_{\Omega}e^{u}v^{2\alpha-1}\;dx &\le \left(\int_{\Omega} e^{{\frac{u}{2}}}v^{2\alpha}\;dx\right)^{\frac{2\alpha-1}{2\alpha}}\left(\int_{\Omega} e^{{\frac{u}{2}}}e^{\alpha u}\;dx\right)^{\frac1{2\alpha}}\quad\text{ and }\\
\int_{\Omega}e^{\alpha u}v\;dx &\le \left(\int_{\Omega} e^{{\frac{u}{2}}}v^{2\alpha}\;dx\right)^{\frac{1}{2\alpha}}\left(\int_{\Omega} e^{{\frac{u}{2}}}e^{\alpha u}\;dx\right)^{\frac{2\alpha-1}{2\alpha}}.\\
\end{align*}
Plugging these inequalities in \eqref{2162} and \eqref{2161} respectively, we deduce that
\begin{align*}
\left(\int_{\Omega}e^{{\frac{u}{2}}}v^{2\alpha}\;dx\right)^\frac1{2\alpha} &\le \frac{\alpha^2}{2\alpha-1}\left(\int_{\Omega} e^{{\frac{u}{2}}}e^{\alpha u}\;dx\right)^{\frac{1}{2\alpha}}\quad\text{ and }\\
\int_{\Omega}e^{{\frac{u}{2}}}\left(e^{\frac{\alpha u}{2}}-1\right)^2\;dx &\le \frac\alpha4\left(\int_{\Omega} e^{{\frac{u}{2}}}v^{2\alpha}\;dx\right)^{\frac{1}{2\alpha}}\left(\int_{\Omega} e^{{\frac{u}{2}}}e^{\alpha u}\;dx\right)^{\frac{2\alpha-1}{2\alpha}}.\\
\end{align*}
Multiplying these inequalities, it follows that
$$
\int_{\Omega}e^{{\frac{u}{2}}}\left(e^{\frac{\alpha u}{2}}-1\right)^2\;dx \le \frac{\alpha^3}{8\alpha-4}\int_{\Omega}e^{(\alpha+{\frac12})u}\;dx
$$
so that
$$
\left(1 - \frac{\alpha^3}{8\alpha-4}\right)\int_{\Omega} e^{{\frac{u}{2}}}e^{\alpha u}\;dx\le 2
\int_{\Omega}e^{\frac{\alpha+1}{2}u}\;dx.
$$
Apply H\"older's inequality to the right-hand side. Then,
$$
\int_{\Omega}e^{\frac{\alpha+1}{2}u}\;dx\le\left(\int_{\Omega}e^{\frac{2\alpha+1}{2}u}\;dx\right)^{\frac{\alpha+1}{2\alpha+1}}\left\vert\Omega\right\vert^{\frac{\alpha}{2\alpha+1}}
$$
and so
\begin{equation} \label{2163}
\left(1 - \frac{\alpha^3}{8\alpha-4}\right)\left(\int_{\Omega} e^{\frac{2\alpha+1}{2}u}\;dx\right)^{\frac{\alpha}{2\alpha+1}}\le 2\left\vert\Omega\right\vert^{\frac{\alpha}{2\alpha+1}}.
\end{equation}
Let $\alpha^*>5/2$ denote the largest root of the polynomial $X^3-8X+4$. We have just proved that $e^u$ is uniformly bounded in $L^p(\Omega)$ for every $p<p^*=\alpha^*+{\frac12}$. In particular, $e^u$ is bounded in $L^p(\Omega)$ for some $p>3$. Using elliptic regularity applied to \eqref{bi}, this implies that $u^*$ is bounded, hence smooth whenever $N\le12$. \hfill\qed

\section{Growth of the $L^1$ norm}
This section and the next provide preparatory results that will be used both for the proof of the Liouville-type theorem and the partial regularity result.
We begin with the case of $\R^N$.
\begin{lemma}\label{l1v}
Assume $N\geq 5$ and let $u$ be a solution of \eqref{gelfandbe} which is stable (resp. stable outside the ball $\super{B_{R_{0}}}$). Let $v=-\Delta u$, $\super v$ its spherical average, and assume that $\super v(\infty)=0$. Let $B_{R}$ denote the ball of radius $R$ (resp. the annulus of radii $R$ and $2R$). Then, there exists a constant $C$ depending only on $N$ (resp. on $N, u, R_{0}$) such that
\begin{equation}\label{l1ev}
\int_{A_{R}} v \; dx\le C R^{N-2}\quad\text{for every $R>0$ (resp. $R>R_{0}$).}
\end{equation}
\end{lemma}
\begin{proof}
We claim that $v>0$. Since the equation is invariant under translation, it suffices to prove that $v(0)>0$. Let $\super v$ be the spherical average of $v$ and assume by contradiction that $v(0)=\ov(0)\leq 0$.
We have
\begin{equation}\label{avsystem}
\left\{
\begin{array}{l l l}
-\Delta \ov=\overline{e^u} & \mbox{ in } \R^N,\\
-\Delta \ou=\ov &  \mbox{ in } \R^N.
\end{array}
\right.
\end{equation}
In particular, $-\Delta \ov=-r^{1-N}(r^{N-1}v')'> 0$, and so $\ov$ is a decreasing function of $r>0$.
Since we assumed that $\ov(0)\leq 0$, it follows that $ \ov(r)< 0$ for all $r>0$.
So, $\ou$ is subharmonic. It follows that $\ou$ is an increasing function of $r>0$. In particular, $\ou$ is bounded below and so, given $R>2R_{0}$,
\begin{equation} \label{yom} \int_{B_{2R}\setminus B_{R}} e^{\ou}\;dx \geq e^{u(0)}\int_{B_{2R}\setminus B_{R}}\;dx \gtrsim  R^N.\end{equation}
Take a cut-off function $\var\in C^2_{c}(B_{4}\setminus B_{{\frac12}})$, with $0\le\var\le1$ in $\R^N$ and $\var\equiv1$ in $B_2\setminus B_{1}$.
Applying Jensen's inequality and the stability inequality \eqref{stability} with test function $\var(x/R)$, we find
$$\int_{B_{2R}\setminus B_{R}} e^{\ou} dx \leq \int_{B_{2R}\setminus B_{R}} \overline{e^{u}}dx \lesssim R^{N-4}$$
which contradicts \eqref{yom},  if $R$ is chosen large enough. Hence, $v(0)>0$.
Now, if $u$ is stable, take $r>0$, take a standard cut-off function $\varphi\in C^2_{c}(B_{2})$ such that $0\le\varphi\le1$ and $\varphi=1$ in $B_{1}$. Apply the stability inequality \eqref{stability} with test function $\varphi(x/r)$ to get
\begin{equation} \label{ama0}
\int_{B_{r}}e^{u}\;dx\lesssim r^{N-4}.
\end{equation}
If $u$ is only stable outside $\super{B_{R_{0}}}$, fix $r>2R_{0}$ and take a cut-off function $\varphi\in C^2_{c}(\R^N)$ such that $\varphi\equiv 0$ in ${B_{R_{0}}}$, $\varphi\equiv 1$ in $B_{r}\setminus B_{R_{0}+1}$, $\varphi\equiv 0$ outside $B_{2r}$, and $\vert\Delta\var\vert\lesssim r^{-2}$ in $B_{2r}\setminus B_{r}$. Using the stability inequality \eqref{stability} with this test function we find that
\begin{equation} \label{ama}
\int_{B_{r}}e^{u}\;dx= \int_{B_{R_{0}+1}}e^{u}\;dx + \int_{B_{r}\setminus B_{R_{0}+1} }e^{u}\;dx\lesssim 1 + r^{N-4}\lesssim r^{N-4}.
\end{equation}
Now, we rewrite the first equation in \eqref{avsystem} as
$$-(r^{N-1}\ov')'=r^{N-1}\overline{e^u}.$$
Integrate on $(0,r)$. By \eqref{ama0}, which holds for every $r>0$ if $u$ is stable (resp. by \eqref{ama}, which holds for $r>2R_{0}$ if $u$ is stable outside $\super{B_{R_{0}}}$),
$$-r^{N-1}\ov'(r)=\int_0^r t^{N-1}\overline{e^u}dt\lesssim r^{N-4}.$$
We integrate once more between $R$ and $+\infty$. Since $\super v(\infty)=0$, we obtain
$$\ov(R)\lesssim R^{-2}.$$
Clearly, \eqref{l1ev} follows.
\end{proof}
Now, we turn to an analogous result on bounded domains. {Let us recall that the extremal solution
is stable.}

\begin{lemma}\label{30juin2012}
Let $N\ge3$ and let $u$ be the extremal solution of \eq{bi}. 
Set $v=-\Delta u$.
Fix $x_{0}\in\Omega$ and let $R_{0}=\text{\rm dist}(x_{0},\partial\Omega)/2$. Then, there exists a constant $C>0$ depending only on $N$, $\Omega$, and $R_{0}$ such that for all $r<R_{0}$,
\neweq{3061}
\int_{B_r(x_{0})} v\;dx\leq C r^{N-2}.
\endeq
\end{lemma}

\proof Without loss of generality, we may assume that $x_{0}=0$ and $\lambda^*=1$.

\noindent{\bf Step 1.} There exists a constant $C=C(N,R_{0})$ such that for all $r\in(0,R_{0})$,
\begin{equation} \label{3062}
\int_{[r<\vert x\vert<R_{0}]}e^{u}\vert x\vert^{2-N}\;dx\le C r^{-2}.
\end{equation}
To see this, consider the function $\psi:B_{2R_{0}}\to\R^N$ given by
\begin{equation} \label{3063}
\psi(x)=\left\{
\begin{aligned}
a-b\vert x\vert^2&\quad\text{if $\vert x\vert<r$,}\\
\vert x\vert^{1-N/2}&\quad\text{if $r\le\vert x\vert<2R_{0}$,}
\end{aligned}
\right.
\end{equation}
where the constants $a=\frac{N+2}{4}r^{1-N/2}$ and $b=\frac{N-2}{4}r^{-1-N/2}$ are chosen so that $\psi$ is $C^1$ and piecewise $C^2$. Take a standard cut-off function
\begin{equation} \label{standard}
\zeta\in C^2_{c}(B_{2}) \quad\text{such that $0\le\zeta\le1$ and $\zeta=1$ on $B_{1}$},
\end{equation}
so that $\varphi(x)=\psi(x)\zeta(x/R_{0})$ belongs to $H^1_{0}(\Omega)\cap H^2(\Omega)$.

Since $u$ is stable, we have
\begin{equation} \label{3064}
\int_{[r<\vert x\vert<R_{0}]}e^{u}\vert x\vert^{2-N}\;dx\le \int_{\Omega}e^{u}\varphi^2\;dx \le \int_{\Omega}\vert\Delta\varphi\vert^2\;dx.
\end{equation}
By \eqref{3063},
\begin{equation} \label{3065}
\int_{[\vert x\vert<r]}\vert\Delta\varphi\vert^2\;dx\le 4N^2 b^2\int_{[\vert x\vert<r]}\;dx\le Cb^2r^N\le C' r^{-2}.
\end{equation}
Similarly,
\begin{equation} \label{3066}
\int_{[r<\vert x\vert<R_{0}]}\vert\Delta\varphi\vert^2\;dx\le C\int_{[r<\vert x\vert]} \vert x\vert^{-2-N}\;dx\le C' r^{-2}.
\end{equation}
Finally,
\begin{equation} \label{3067}
\int_{[R_{0}<\vert x\vert]\cap\Omega}\vert\Delta\varphi\vert^2\;dx\le C(N,R_{0})\le C' r^{-2}.
\end{equation}
Collecting \eqref{3064},  \eqref{3065}, \eqref{3066}, and \eqref{3067}, the estimate \eqref{3062} follows.

\noindent{\bf Step 2.} Take as before a standard cut-off $\zeta$ satisfying \eqref{standard} and let
$$
\psi(x) = \int_{\R^N} \vert x-y\vert^{2-N}\zeta(y)\;dy,\qquad\text{for $x\in\R^N$.}
$$
Then, there exists a constant $C$, depending on $N$ only such that
\begin{equation} \label{3068}
\psi(x)\le C\min(1,\vert x\vert^{2-N}), \qquad\text{for $x\in\R^N$.}
\end{equation}
Indeed, if $\vert x\vert>4$ and $\vert y\vert<2$,
$$
\vert x-y\vert\ge \vert x\vert - \vert y\vert \ge \frac12\vert x\vert
$$
so that
$$
\psi(x)\le 2^{N-2}\vert x\vert^{2-N}\int_{\R^N}\zeta(y)\;dy\le C\vert x\vert^{2-N},
$$
while if $\vert x\vert\le4$ and $\vert y\vert<2$,
$$
\vert x-y\vert\le \vert x\vert + \vert y\vert \le 6
$$
and so
$$
\psi(x)\le \| \zeta \|_{L^\infty(\R^N)}\int_{\vert x-y\vert<6} \vert x-y\vert^{2-N}\;dy \le C.
$$
\eqref{3068} follows.

\noindent{\bf Step 3.} Take $r\in(0,R_{0})$, $\zeta$ a standard cut-off function satisfying \eqref{standard}, and let $\varphi_{r}$ be the solution to
\begin{equation} \label{3069}
\left\{
\begin{aligned}
&-\Delta \varphi_{r}=\zeta(x/r)&&\quad\mbox{ in }\Omega,\\
&\varphi_{r}=0&&\quad\mbox{ on }\partial\Omega.
\end{aligned}
\right.
\end{equation}
Then, there exists a constant $C$ depending on $N$ only such that
\begin{equation} \label{3069a}
\varphi_{r}(x)\le C r^2 \min\left(1,r^{N-2}\vert x\vert^{2-N}\right), \quad\text{for all $x\in\Omega$.}
\end{equation}
This easily follows from the maximum principle, observing that a constant multiple of $r^2\psi(x/r)$ is a supersolution to the above equation.

\noindent{\bf Step 4.} There exists a constant $C$ depending on $N$ and $\Omega$ only, such that
\begin{equation} \label{3069p}
\int_{\Omega}e^{u}\;dx\le C.
\end{equation}
This is an obvious consequence of Equation \eqref{2163} (which holds for any $\alpha\in({\frac12},5/2]$) and H\"older's inequality.

\noindent{\bf Step 5.} Multiply \eqref{3069} by $v=-\Delta u$ and integrate by parts. Then,
$$
\int_{[\vert x\vert<r]}v\;dx \le \int_{\Omega}v\zeta(x/r)\;dx = \int_{\Omega}e^{u}\varphi_{r}\;dx.
$$
Using Step 3, we have
$$
\int_{[\vert x\vert<r]}e^{u}\varphi_{r}\;dx \le C r^2 \int_{[\vert x\vert<r]}e^{u}\;dx.
$$
Using stability with test function $\zeta(x/r)$, we also have
$$
\int_{[\vert x\vert<r]}e^{u}\;dx \le C r^{N-4}
$$
and so
$$
\int_{[\vert x\vert<r]}e^{u}\varphi_{r}\;dx \le C r^{N-2}.
$$
By Step 1 and Step 3,
$$
\int_{[r<\vert x\vert<R_{0}]}e^{u}\varphi_{r}\;dx \le C r^{N} \int_{[r<\vert x\vert<R_{0}]}e^{u}\vert x\vert^{2-N}\;dx\le C r^{N-2}.
$$
Finally, using \eqref{3069p},
$$
\int_{[R_{0}<\vert x\vert]\cap\Omega}e^{u}\varphi_{r}\;dx \le C R_{0}^{2-N}r^N \int_{\Omega}e^{u}\;dx\le C' r^{N-2}.
$$
\eqref{3061} follows.
\hfill\qed

\section{A bootstrap argument} \label{sec boot}
Our next task consists in improving the $L^1$-estimates of the previous section to $L^p$-estimates for larger values of $p$. We do this through a bootstrap argument which is reminescent of the classical Moser iteration method, up to one major difference: we will take advantage of both  the standard Sobolev inequality and the stability inequality \eqref{stability}. To be more precise, rather than \eqref{stability}, the following interpolated version of it will be used.
\begin{lemma}\label{lem interp}
Let $\Omega$ be an open set of $\R^N$, $N\ge1$. Assume that the stability inequality \eqref{stability} holds. Then, for every $s\in[0,1]$,
$$
\int_{\Omega} e^{su}\varphi^2\;dx\leq \int_{\Omega} |(-\Delta)^{s/2} \varphi|^2\;dx\quad\mbox{ for all }\varphi\in H^2(\Omega)\cap H^1_{0}(\Omega).
$$
In particular, for $s={\frac12}$,
\neweq{stabrev}
\int_{\Omega} e^{{\frac{u}{2}}}\varphi^2\;dx\leq \int_{\Omega} |\nabla \varphi|^2\;dx\quad\mbox{ for all }\varphi\in H^1_0(\Omega).
\endeq	
\end{lemma}

\noindent {\bf Proof} Consider first the case where $\Omega=\R^N$. We apply complex interpolation between the family of spaces $X_s$, $Y_s$ given for $0\leq s\leq 1$ by
$$
X_s=L^2((2\pi)^{-N}|\xi|^{4s}d\xi)\,, \quad Y_s=L^2(e^{su}dx).
$$
Recall that the inverse Fourier transform ${\mathcal F}^{-1}:X_0\rightarrow Y_0$ satisfies $\|{\mathcal F}\|_{{\mathcal L}(X_0,Y_0)}=1$. Furthermore, by the stability inequality \eqref{stability}  and Plancherel's theorem, we have
$$
\int_{\R^N} e^{u} \varphi^2\;dx\leq \int_{\R^N} |\Delta \varphi |^2\;dx=(2\pi)^{-n}\int_{\R^N} |\xi|^4|  |{\mathcal F}(\phi)|^2 \;d\xi.
$$
Thus, ${\mathcal F}^{-1}:X_1\rightarrow Y_1$ satisfies $\|{\mathcal F}\|_{{\mathcal L}(X_1,Y_1)}\leq 1$. By \cite[Theorem 2]{stein}, we deduce that
$\|{\mathcal F}\|_{{\mathcal L}(X_s,Y_s)}\leq 1$ for all $0\leq s\leq 1$. 

In the case where $\Omega$ is a bounded open set, simply repeat the above proof, using the spectral decomposition of the Laplace operator in place of the Fourier transform. More precisely, let $(\lambda_{k})$ denote the eigenvalues of the Laplace operator (with domain $D=H^2(\Omega)\cap H^1_{0}(\Omega)$), let $\hat\varphi_{k}$ be the $k$-th component of $\varphi$ in the Hilbert basis of eigenfunctions associated to $(\lambda_{k})$ and interpolate between the family of weighted $L^2$-spaces $X_{s}$, $Y_{s}$ corresponding to the norms
$$
\| \varphi \|_{X_{s}}^2 = \sum_{k} \lambda_{k}^{2s}\vert\hat\varphi_{k}\vert^2, \quad \| \varphi \|_{Y_{s}}^2 = \int_{\Omega}e^{su} \varphi^2\;dx.
$$
In the case where $\Omega$ is an unbounded proper open set,  take $k>0$, let $\Omega_{k}=\Omega\cap B_{k}$ and
$$
-\mu_{k}=\inf\left\{ \int_{\Omega_{k}}\left(\vert\Delta\varphi\vert^2-e^{u}\varphi^2\right)\;dx \; : \;
\varphi\in H^2(\Omega_{k})\cap H^1_{0}(\Omega_{k}), \| \varphi \|_{L^2(\Omega_{k})}=1
\right\}.
$$
Then, the previous analysis leads to
$$
\int_{\Omega_{k}} \left[\left(e^{u}-\mu_{k}\right)_{+}\right]^s\varphi^2\;dx\leq \int_{\Omega_{k}} |(-\Delta)^{s/2} \varphi|^2\;dx\quad\mbox{ for all }\varphi\in H^2(\Omega_{k})\cap H^1_{0}(\Omega_{k}).
$$
By \eqref{stability}, $\lim_{k\to+\infty}-\mu_{k}\ge0$ and the result follows.
\hfill\qed

Our next lemma is simply the first step in the Moser iteration method: we multiply the equation by a power of its right-hand side, localize, and integrate.
\begin{lemma}\label{fe} Let $\Omega$ be an open set of $\R^N$, $N\ge1$. Assume $(u,v)\in C^2(\Omega)^2$ solves
$$\left\{
\begin{aligned}
-&\Delta u=v&&\quad\mbox{ in }\Omega,\\
-&\Delta v=e^u&&\quad\mbox{ in }\Omega.
\end{aligned}
\right.
$$
Take $\a> \frac12$, $\var\in C^1_c(\Omega)$. Then, there exists a constant $C$ depending on $\alpha$ only, such that
\begin{equation}\label{1}
\frac{\sqrt{2\a-1}}{\a}||\nabla(v^\a\var)||_{L^2(\Omega)}\leq \|e^{\frac u2}v^{\alpha-\frac12}\var\|_{L^2(\Omega)}+C||v^\a\nabla\var||_{L^2(\Omega)}.
\end{equation}
and
\begin{equation}\label{2}
\frac{2}{\sqrt{\a}}||\nabla(e^{\frac{\a}{2}u}\var)||_{L^2(\Omega)} \leq \|e^{\frac\alpha2 u}v^{\frac12}\var\|_{L^2(\Omega)} + C ||e^{\frac{\a}{2}u}\nabla\var||_{L^2(\Omega)}.
\end{equation}
\end{lemma}
\begin{proof}
Since the computations are very similar, we prove only \eqref{1}.  We multiply $-\Delta v=e^u$ by $v^{2\a-1}\var^2$ and we integrate. We obtain
\begin{equation*}
\begin{split}
\int_{\Omega} e^uv^{2\a-1}\var^2\;dx
&=\int_{\Omega}\nabla v\cdot\nabla(v^{2\a-1}\var^2)\;dx\\
&=\frac{2\a-1}{\a^2}\int_{\Omega}|\nabla v^\a|^2\var^2\;dx+2\int_{\Omega} v^{2\a-1}\varphi\nabla v\cdot \nabla \varphi\; dx\\
&= \frac{2\a-1}{\a^2}\left(\int_{\Omega} |\nabla(v^\a\var)|^2\;dx-\int_{\Omega} v^{2\a}|\nabla\var|^2\;dx\right)\\ &\quad-\frac{2(\a-1)}{\a^2}\int_{\Omega} v^{\a}\var\nabla\var\cdot \nabla v^\a \;dx.
\end{split}
\end{equation*}
In the last term of the right hand side, replace  $\var\nabla v^\a$ by $\nabla(v^\a\var)-v^\a\nabla\var$. Then,
\begin{equation*}
\begin{split}
\int_{\Omega} e^uv^{2\a-1}\var^2\;dx&=\frac{2\a-1}{\a^2}\int_{\Omega} |\nabla(v^\a\var)|^2\;dx-\frac{1}{\a^2}\int_{\Omega} v^{2\a}|\nabla\var|^2\;dx\\
&\quad-\frac{2(\a-1)}{\a^2}\int_{\Omega} v^{\a}\nabla\var\nabla (v^\a\var)\;dx,
\end{split}
\end{equation*}
which we rewrite as
\begin{equation}\label{a}
\begin{split}
({2\alpha-1}) \int_{\Omega} \vert\nabla(v^\a\var)\vert^2\;dx& = {2(\a-1)}\int_{\Omega} v^{\a}\nabla\var\nabla (v^\a\var)\;dx +\\&\quad \int_{\Omega} v^{2\a}|\nabla\var|^2\;dx+\a^2\int_{\Omega} e^uv^{2\a-1}\var^2\;dx.
\end{split}
\end{equation}
By the Cauchy-Schwarz inequality,
$$\left|\int_{\Omega} v^{\a}\nabla\var\nabla (v^\a\var)\;dx\right|\leq \left(\int_{\R^n} |\nabla(v^\a\var)|^2\;dx\right)^\frac12\left(\int_{\Omega} v^{2\a}|\nabla\var|^2\;dx\right)^\frac12,$$
Plugging this in \eqref{a}, we obtain a quadratic inequality of the form
$$({2\a-1})X^2\leq 2|\a-1|AX+A^2+B^2$$
where
$$X=||\nabla(v^\a\var)||_2, \quad \ A=||v^\a\nabla\var||_2
\quad\text{and}\quad
B=\alpha\|  e^{\frac12 u}v^{\a-\frac12}\var\|_{2}.$$
Solving the quadratic inequality, we deduce that
$$X\leq \frac{\vert\alpha-1\vert A+\sqrt{\vert\a-1\vert^2A^2+(2\a-1)(A^2+B^2)}}{2\a-1}\le \frac{B}{\sqrt{2\a-1}}+ C_{\a}A.$$
\eqref{1} follows by replacing $A,B$ and $X$ with their values.
\end{proof}
We have just used the equation. Now, we use the stability assumption.
\begin{lemma}\label{se}
Make the same assumptions as in Lemma \ref{fe}. Assume in addition that \eqref{stabrev} holds for every $\varphi\in C^1_{c}(\Omega)$. Let $\alpha^\sharp, \alpha^*$ denote the largest two roots of the polynomial $X^3-8X+4$.Then, for every $\alpha\in(\alpha^\sharp,\alpha^*)$, there exists a constant $C$ depending on $\alpha$ only such that
\begin{equation}\label{3}
\int_{\Omega}|\nabla(v^\a\var)|^2\;dx \le C \int_{\Omega}v^{2\a}|\nabla\var|^2\;dx,
\end{equation}
or
\begin{equation}\label{4}
\int_{\Omega}|\nabla(e^{\frac{\a}{2}u}\var)|^2\;dx\le C \int_{\Omega}e^{\a u}|\nabla\var|^2\;dx,
\end{equation}
for any $\var\in C^1_c(\Omega)$.
\end{lemma}
\begin{proof}
By H\"older's inequality,
$$\int_{\Omega} e^uv^{2\a-1}\var^2\;dx\leq \left(\int_{\Omega}e^{\frac{2\alpha+1}{2}u}\var^2\;dx\right)^{\frac1{2\a}}\left(\int_{\Omega}e^{\frac u2}v^{2\a}\var^2\;dx\right)^{\frac{2\a-1}{2\a}}.$$
Using the stability inequality \eqref{stabrev}, we deduce that
$$\int_{\Omega} e^uv^{2\a-1}\var^2\;dx \leq H^{\frac{1}{\a}}K^{2-\frac{1}{\a}},$$
where we set
$$H=||\nabla(e^{\frac{\a}{2}u}\var)||_2 \ \mbox{ and }\  K=||\nabla(v^\a\var)||_2.$$
Similarly,
$$\int_{\Omega} e^{\a u}v\var^2\;dx \leq K^{\frac{1}{\a}}H^{2-\frac{1}{\a}}.$$
Combining with \eqref{1}-\eqref{2}, this gives
\begin{equation}\label{5}
\frac{\sqrt{2\a-1}}{\a}H \leq K^{\frac{1}{2\a}}H^{1-\frac{1}{2\a}}+ C ||e^{\frac{\a}{2}u}\nabla\var||_2,
\end{equation}
\begin{equation}\label{6}
\frac{2}{\sqrt\a}K \leq H^{\frac{1}{2\a}}K^{1-\frac{1}{2\a}}+ C ||v^\a\nabla\var||_2.
\end{equation}
Multiply \eqref{5} by \eqref{6}. Then,
\begin{equation}\label{7}
\left(\frac{2\sqrt{2\a-1}}{\a\sqrt\a}-1\right)HK\leq a H^{\frac{1}{2\a}}K^{1-\frac{1}{2\a}} 
+b K^{\frac{1}{2\a}}H^{1-\frac{1}{2\a}} +ab,
\end{equation}
where
$$a=C ||e^{\frac{\a}{2}u}\nabla\var||_2 \mbox{ and } b=C ||v^\a\nabla\var||_2.$$
Note that for $\alpha\in(\alpha^\sharp, \alpha^*)$,
$$\delta:=\frac{2\sqrt{2\a-1}}{\a\sqrt\a}-1 >0.$$
Introduce
$$
X=K^{\frac{1}{2\a}}H^{1-\frac{1}{2\a}} \ \mbox{ and }\  Y=H^{\frac{1}{2\a}}K^{1-\frac{1}{2\a}}.
$$
Then, \eqref{7} can be rewritten as
\begin{equation*}
\delta XY\leq aY+bX+ab,
\end{equation*}
and so, either $X$ is bounded by a multiple of $a$ or  $Y$ by a multiple of $b$.
In the former case, recalling \eqref{5}, we obtain \eqref{4}. In the latter case, \eqref{6} implies \eqref{3}.
\end{proof}
In the two previous lemmata, we have used successively the equation and the stability assumption. Now, we apply the Sobolev inequality to set up a bootstrap procedure.
\begin{lemma}\label{boot}
Make the same assumptions as in Lemma \ref{se}. Take $\alpha\in(\alpha^\sharp,\alpha^*)$ and for $R>0$, let $B_{R}$ denote a ball of radius $R$ (resp. an annulus of radii $R$ and $R/2$) contained in $\Omega$. Assume that there exists a constant $C$ depending on $N$ and $\a$ only, such that for all $R>0$ (resp. for all $R$ large enough)
\begin{equation}\label{Ha}
\int_{B_R} (e^{\a u}+v^{2\a})\;dx \le C R^{N-4\a}. \tag{$H_\a$}
\end{equation}
Then, $\left(H_{\frac{N}{N-2}\a}\right)$ also holds.
\end{lemma}
\noindent Bootstrapping the above lemma, we find
\begin{corollary}\label{mondaycor}
{Make the same assumptions as in Lemma \ref{se}.}
Assume that \eqref{Ha} holds for some $\alpha\in(\alpha^\sharp,\alpha^*)$. Then,
\begin{equation} \label{capacitary estimate}
\begin{aligned}
\int_{B_R} e^{p u}\;dx &\le C R^{N-4p},\qquad\text{for all
$p<p^*:=\alpha^*+\frac12$,}\\\
\int_{B_R} v^{q}\;dx &\le C R^{N-2q},\qquad\text{for all
$q<q^*:=\frac {2N}{N-2}\alpha^*$.}\\
\end{aligned}
\end{equation}
\end{corollary}

\begin{remark}The inequality \eqref{Ha} can be further simplified if the boundary values/limiting behavior at infinity of the solution is known. See in particular Proposition \ref{ogps}.
\end{remark}

\proofp{of the lemma \ref{boot}}
Assume \eqref{Ha} is valid. By Lemma \ref{se}, either \eq{3} or \eq{4} holds.\\
Assume that \eqref{3} is valid (the other case is similar). Using the Sobolev embedding, we obtain
\begin{equation*}
\left(\int_{\R^N} v^{2^*\a}\var^{2^*}dx\right)^{\frac{2}{2^*}}\le \int_{\R^N} |\nabla(v^{\a}\var)|^2dx \lesssim \int_{\R^N} v^{2\a}|\nabla \var|^2dx.
\end{equation*}
Take a standard cut-off function $\psi\in C^1_{c}(B_{2})$ such that $0\le\psi\le1$, $\psi=1$ in $B_{1}$, and $\psi=0$ outside $B_{2}$. Apply the above inequality with $\var(x)=\psi(x/R)$ and use \eqref{Ha}.
Then,
\begin{equation}\label{Ha1}
\int_{B_R} v^{2^*\a}dx\lesssim R^{N-2.2^*\a}.
\end{equation}
Going back to \eqref{5}, we deduce similarly that
\begin{equation}\label{Ha2}
\int_{B_R} e^{\frac{2^*\a}{2}u}dx\lesssim R^{N-2^*.2\a}.
\end{equation}
%
\hfill\qed

\proofp{of the corollary \ref{mondaycor}} By H\"older's inequality, if \eqref{Ha} holds for some $\alpha$, then $(H_{\beta})$ holds for all $\beta\le\alpha$. So, bootstrapping \eqref{Ha}, we easily deduce that it holds for all $\alpha<\frac{N}{N-2}\alpha^*$. Fix at last $\alpha<\alpha^*$ and apply now stability \eqref{stabrev}, with test function $e^{\frac \a2 u}\psi(x/R)$, to deduce that \eqref{capacitary estimate}  holds for all $p=\alpha+\frac12<p^*=\alpha^*+\frac12$.
\hfill\qed

\section{The Liouville theorem}
We prove here Theorem \ref{liouville}.  Assume by contradiction that there exists a solution $u$ of \eqref{gelfandbe} which is stable outside a compact set {$K\subset B_{R_0}$} and such that $\super v(\infty)=0$.

\noindent{\bf Step 1. } $e^{u}\in L^p(\R^N)$, for every $p<p^*$.

By Lemma \ref{l1v}, we have
\begin{equation} \label{lt1}
\int_{{ A_{R}}} v \; dx\le C R^{N-2}\quad\text{for every $R>R_{0}$,}
\end{equation}
where { $A_{R}$} is the annulus of radii $R$ and $2R$. In addition, stability \eqref{stability} implies that
\begin{equation} \label{lt2}
\int_{ {A_{R}}} e^{u} \; dx\le C R^{N-4}\quad\text{for every $R>R_{0}$,}
\end{equation}
Recall now the following standard elliptic estimate : for $p\in [1,\frac{N}{N-2})$,
\begin{equation*}
||v||_{L^p({ B_2\backslash B_1})}\lesssim ||\Delta v||_{L^1({ B_2\backslash B_1})}+||v||_{L^1({ B_4\backslash B_{1/2}})},
\end{equation*}
and its rescaled version
\begin{equation*}
R^{-\frac{N}{p}}||v||_{L^p({ A_R})}\lesssim R^{-N}\left(R^2||\Delta v||_{L^1({ A_R})}+||v||_{L^1({ B_{4R}\backslash B_{R/2}})}\right).
\end{equation*}
Applying this estimate respectively to $u$ and $v$, we deduce from \eqref{lt1}, \eqref{lt2} that
\eqref{Ha} holds for any $\alpha\in[1,\frac{N}{N-2})$ and all large $R$.  Hence, \eqref{capacitary estimate} holds. By a straightforward covering argument, Step 1 follows.

The rest of the proof is very similar to the one given in \cite{df};

\noindent{\bf Step 2. }$$\lim_{\vert x\vert\to+\infty}\vert x\vert^4e^u=0.$$
{By Step 1., given any $\delta>0$, we can choose $\tilde R$ large enough such that
\begin{equation}\label{Br}
\int_{|y|\geq \tilde R} e^{pu}dx \leq \delta^{p}.
\end{equation}
Let $x\in \R^N\backslash \overline{B_{R_0}}$, $|x|\geq 4\tilde R$. Set $R=\frac{2}{3}|x|$. This yields
$$B(x,R/4)\subset A_R= B_{2R}\backslash B_R\subset \{y\in \R^N\ |\ |y|\geq \tilde R\}.$$
Thus, we have
\begin{equation} \label{ce}
\begin{aligned}
\int_{B(x,R/4)} e^{p u}\;dx &\le C R^{N-4p},\qquad\text{for all
$p<p^*:=\alpha^*+\frac12$,}\\\
 \int_{B(x,R/4)} v^{q}\;dx &\le C R^{N-2q},\qquad\text{for all
$q<q^*:=\frac {2N}{N-2}\alpha^*$.}\\
\end{aligned}
\end{equation}
Next fix $\delta>0$ and consider $w=e^{u}$. By Kato's inequality, $w$ satisfies
\begin{equation}\label{st}
-\Delta w - v w\le 0\quad\text{in $\R^N$.}
\end{equation}
Take $\varepsilon$ small enough such that $\frac{N}{2-\varepsilon}<\frac{2N}{N-2}\a^*$. Here, we have used the assumption $5\leq N\leq 12$. The Serrin-Trudinger inequality \cites{serrin,trudinger} for subsolutions to \eqref{st} ensures that for any $p<p^*$
\begin{equation*}
||w||_{L^\infty(B(x,\frac{R}{8}))}\leq C R^{-\frac Np}||w||_{L^{p}(B(x,\frac{R}{4}))}
\end{equation*}
where $C$ depends to $N$, $p$ and
$$R^{\varepsilon}||v||_{L^{\frac{N}{2-\varepsilon}}(B(x,\frac{R}{4}))}.$$
In particular, for $p=\frac N4$ and using \eqref{ce}
\begin{equation}\label{eux}
e^{u(x)}\leq C R^{-4}||e^u||_{L^{N/4}(B(x,\frac{R}{4}))}
\end{equation}
where $C$ depends only to $N$. Combining \eqref{Br} and \eqref{eux} gives
\begin{equation*}
e^{u(x)}\leq C\delta R^{-4}\lesssim \delta |x|^{-4},
\end{equation*}
which proves Step 2.

\noindent{\bf Step 3. } By Step 2., there exists $R_1>R_0$ such that
$$-\Delta \ov \leq \frac{1}{2r^4} \quad \mbox{for all }r> R_1.$$
Hence,
$$\ov'(r)\geq \frac{C(N)}{r^{N-1}}-\frac{1}{2(N-4)r^3} \quad \mbox{for all }r> R_1.$$
Integrating between $r$ and $+\infty$, this yields
$$-\Delta \ou (r)=\ov(r)=\ov(r)-\ov(\infty)\leq \frac{1}{2r^2}\left(\frac{1}{2(N-4)}-\frac{C'(N)}{r^{N-4}}\right) \quad \mbox{for all }r> R_1.$$
Since $N\geq 5$ and choosing $R_2>R_1$ large enough to have
$$\frac{1}{2(N-4)}-\frac{C'(N)}{r^{N-4}}\leq 1 \mbox{ for all } r> R_3,$$
we get
$$-\Delta \ou\leq \frac{1}{2r^2} \mbox{ for all }r> R_3.$$
In the same way, we have for some $R_3> R_2$
$$\ou'(r) \geq-\frac 1 r \quad \mbox{for all }r> R_3.$$
Integrating the latter and taking the exponential, we get
$$r^4 e^{\ou (r)}\geq cr^3 \quad \mbox{for all }r> R_3,$$
where the constant $c>0$ does not depend on $r$. By Jensen's inequality, we have for any $r>R_3$,
$$cr^3 \le r^4e^{\ou (r)} \leq r^4 \overline{e^u}(r) \leq \max_{|x|=r} |x|^4e^{u(x)}
.$$
This contradicts Step 2.
\hfill\qed
}

\section{Partial regularity of the extremal solution in dimension $N\ge 13$.}
In this section, we prove the second part of Theorem \ref{th reg}. As in the previous sections,  we interpret our equation as the system \eqref{systembi}. 
Here, our task consists in showing that if a rescaled $L^p$ norm of $e^u$ is small on some ball $B$, then it remains small on any ball of smaller radius, which is included in $B$. This provides an estimate in Morrey spaces, for which an $\epsilon$-regularity theorem is available, thanks to the Moser-Trudinger inequality.

Let $u$ be the extremal solution of \eqref{bi}.
By scaling, we
may assume that $B_1(0)\subset\subset \Omega$. For any $x\in B_{{\frac12}}(0)$,
$1\leq p<p^*$ and $0<r\leq 1-|x|$ we define
\[
E(x,r)=r^{4p-N}\int_{B_r(x)}e^{pu}dy.
\]

\begin{lemma}\label{e}
Assume $E(0,1)<1$ and fix $p<q<p^*$. Then, there exists a positive
constant $C=C(N,p,q)>0$ such that
\begin{equation}\label{este}
E(x,r)\leq C\, E(0,1)^{1-{\frac pq}}\quad\mbox{ for all } x\in B_{{\frac12}}(0),\,
0<r\leq {\frac12}.
\end{equation}
\end{lemma}
\begin{proof}
Let $u=u_1+u_2$, where
\[
\left\{
\begin{aligned}
-&\Delta u_1=0&&\quad\mbox{ in } B_1(0),\\
&u_1=u&&\quad\mbox{ on }\partial B_1(0),
\end{aligned}
\right. \quad\mbox{ and }\quad
\left\{
\begin{aligned}
-&\Delta u_2=v&&\quad\mbox{ in } B_1(0),\\
&u_2=0&&\quad\mbox{ on }\partial B_1(0).
\end{aligned}
\right.
\]
By the maximum principle, $u_1,u_2>0$ in $B_1(0)$. Write $E(x,r)=E_1+E_2$, where
$$
E_1=r^{4p-N}\int_{B_r(x)\cap [u_2\leq \theta]}
e^{pu}\;dy\,\,\,, \quad E_2=r^{4p-N}\int_{B_r(x)\cap [u_2>
\theta]} e^{pu}\;dy,
$$
and where $\theta>0$ will be fixed later on. By Kato's inequality,
$e^{pu_1}$ is subharmonic. By the mean value inequality, it follows that
\begin{equation} \label{onemore}
\fint_{B_r(x)}
e^{pu_1}\;dy \le \fint_{B_{{\frac12}}(x)}
e^{pu_1}\;dy.
\end{equation}
Hence,
\begin{multline}\label{este1}
E_1
\leq r^{4p-N}e^{p\theta} \int_{B_r(x)}
e^{pu_1}\;dy\lesssim  r^{4p}e^{p\theta}\int_{B_{{\frac12}}(x)}
e^{pu}\;dy\\
\leq  r^{4p}e^{p\theta}\int_{B_{1}(0)} e^{pu}\;dy\leq
r^{4p}e^{p\theta}E(0,1).
\end{multline}
To estimate $E_2$, we first note that
\neweq{este2}
\big\vert B_r(x)\cap [u_2>\theta] \big\vert\leq
e^{-p\theta}\int_{B_r(x)\cap [u_2> \theta]} e^{pu}\;dy\leq
\frac{r^N E(x,r)}{r^{4p}e^{p\theta}}.
\endeq
By H\"older's inequality, we also have
\begin{equation} \label{este2b}
\int_{B_r(x)\cap [u_2> \theta]} e^{pu}\;dy \le \left(\int_{B_r(x)}e^{qu}\;dy\right)^{{\frac pq}}\!\!\! \Big |
\big\{B_r(x)\cap [u_2>\theta]
\big\}\Big|^{1-{\frac pq}}.
\end{equation}
Recall from Section \ref{sec boot} that
\[
\int_{B_r(x)} e^{qu}\;dy\lesssim r^{N-4q}.
\]
Using this together with \eq{este2}, \eqref{este2b}, we find
\neweq{este3}
E_2\leq r^{4p-N}\left(\int_{B_r(x)}e^{qu}\;dy
\right)^{{\frac pq}}\!\!\! \Big |
\big\{B_r(x)\cap [u_2>\theta]
\big\}\Big|^{1-{\frac pq}}\lesssim \left[\frac{E(x,r)}{r^{4p} e^{p\theta}}
\right]^{1-{\frac pq}}.
\endeq
Combining \eq{este1} and \eq{este3}, for all $x\in B_{{\frac12}}(0)$ and
$0<r\leq {\frac12}$ we find
\neweq{este4}
E(x,r)\lesssim r^{4p}e^{\theta p}E(0,1)+\left[\frac{E(x,r)}{r^{4p}
e^{p\theta}} \right]^{1-{\frac pq}}.
\endeq
If for some $x\in B_{{\frac12}}(0)$ and $0<r\leq {\frac12}$ we have
\neweq{este5}
r^{4p}E(0,1)\geq \left[\frac{E(x,r)}{r^{4p}} \right]^{1-{\frac pq}},
\endeq
then, using the assumption $E(0,1)<1$, we find
$$
E(x,r)\leq r^{4p{\frac{2q-p}{q-p}}}E(0,1)^{\frac q{q-p}}\le
E(0,1)^{1-{\frac pq}},
$$
which is the desired inequality. So, assume that \eq{este5} does not hold. Then,
we can find $\theta>0$ such that
$$
r^{4p}(e^{p\theta})^{2-{\frac pq}}E(0,1)= \left[\frac{E(x,r)}{r^{4p}}
\right]^{1-{\frac pq}},
$$
that is,
$$
r^{4p}e^{p\theta}E(0,1) = \left[\frac{E(x,r)}{r^{4p}e^{p\theta}}
\right]^{1-{\frac pq}}.
$$
Using this fact in \eq{este4} we derive
$$
E(x,r)\lesssim r^{4p}e^{\theta p}E(0,1)\lesssim \Big[ E(0,1)
E(x,r)\Big]^{\frac{q-p}{2q-p}},
$$
and finally $E(x,r)\lesssim E(0,1)^{1-{\frac pq}}$.
\end{proof}

\noindent{\bf Proof of Theorem \ref{th reg} completed.} Let $\Sigma$ be the singular set of $u$. Let
$1<p<q<p^*$ be fixed.  We claim that
\neweq{sigma1}
\Sigma\subset \left\{ x\in
\Omega: \limsup_{r\rightarrow 0} r^{4p-N}\int_{B_r(x)}
e^{pu}>0\right\}.
\endeq
Assume to the contrary that there exists $x_0\in \Sigma$ such that
$\lim_{r\rightarrow 0} E(x_0,r)=0$. Fix $\varepsilon_0>0$. Then,
there exists $\rho=\rho(\eps_0)>0$ such that $E(x_0,r)\leq \eps_0$
for all $0<r\leq \rho$. For the sake of clarity let us assume that
$x_0=0$ and $\rho=1$. Thus, we can apply Lemma \ref{e} and obtain
$$
E(x,r)\leq C(N,p,q) \eps_0^{1-{\frac pq}}\quad\mbox{ for all } x\in
B_{{\frac12}}(0)\,,0< r\leq {\frac12}.
$$
Using H\"older's inequality this implies that $e^u\in M^{N/4}(B_{{\frac12}}(0))$
and
$$
\|e^u\|_{M^{N/4}(B_{{\frac12}}(0))}\leq C(N,p,q)\eps_0^{{\frac1p}-{\frac1q}}.
$$
Let
$$
w(x)=\int_{\R^N}|x-y|^{4-N}e^{u(y)}dy\,,\quad x\in B_{{\frac12}}(0)
$$
and $\tilde w=u-w$. By Lemma 7.20 in \cite{GT} it follows that
$e^{\beta w}\in L^1(B_{{\frac12}}(0))$ for all
$$
0<\beta<\frac{C(N)}{\|e^{u}\|_{M^{\frac N4}(B_{{\frac12}}(0))}}.
$$
Since $\tilde w$ is biharmonic in $B_{{\frac12}}(0)$, it follows that
$e^u\in L^{\beta}(B_{1/4}(0))$ for all
$$
0<\beta<C(N,p,q) \eps_0^{{\frac1q}-{\frac1p}}.
$$
Letting $\eps_0<<1$ small, we have  $e^u\in
L^\beta(B_{1/4}(0))$ for some $\beta>N/4$, which, by standard
regularity theory, yields $u\in L^\infty(B_{1/8}(0))$ and
contradicts $0\in \Sigma$.

Hence, \eq{sigma1} holds for all $1<p<p^*$. By \cite[Lemma
5.3.4]{book}, it follows that ${\mathcal
H}^{N-4p}(\Sigma)=0$ for all $1<p<p^*$. Thus ${\mathcal
H}_{dim}(\Sigma)\leq N-4p^*$. \qed

\section{A scaling argument}
This last section is devoted to the proof of Theorem \ref{th reg2}.  By rescaling, we may always assume that $\lambda=1$. By standard elliptic regularity, it suffices to show that $u\le C$ in $\Omega$. We assume to the contrary that there exists a sequence of solutions $u_{n}$ of fixed index $k$, such that $M_{n}:=\max_{\super\Omega} u_{n}\to+\infty$, as $n\to+\infty$. Let $x_{n}$ denote a corresponding point of maximum of $u_{n}$.
Passing to a subsequence if necessary, we may assume that there exists  $x_{0}\in\super\Omega$, such that $x_{n}\to x_{0}$, as $n\to+\infty$. Since $\Omega$ is convex, we can also assert that $u_{n}$ is uniformly bounded in a fixed neighborhood of the boundary i.e. $x_{n}\in\omega\subset\subset\Omega$ for large $n$. See e.g. \cite{wei} for this standard boundary estimate.

We use a scaling argument. Let $r_{n}=e^{-M_{n}/4}$ and $U_{n}(x)= u_{n}(x_{n}+r_{n}x)-M_{n}$, for $x\in \Omega_{n}:=\frac1{r_{n}}(\Omega-x_{n})$. Then, $U_{n}$ solves
\begin{equation} \label{fireg4}
\left\{
\begin{aligned}
\Delta^2 U_{n}&= e^{U_{n}}\qquad\text{in $\Omega_{n}$,}\\
U_{n}+M_{n}=\Delta U_{n}&=0\qquad\text{on $\partial\Omega_{n}$.}
\end{aligned}
\right.
\end{equation}
%
%
Since $x_{n}\in\omega\subset\subset\Omega$, $\Omega_{n}\to\R^N$, as $n\to+\infty$. We claim that $(U_{n})$ is uniformly bounded on compact sets of $\R^N$. To see this, fix a ball $B_{R}$ and $n$ so large that $B_{R}\subset\Omega_{n}$. Write $-\Delta U_{n}=V_{n}=V_{n}^1-V_{n}^2$, where $V_{n}^1$ solves
\begin{equation*}
\left\{
\begin{aligned}
-\Delta V_{n}^1&= e^{U_{n}}\qquad\text{in $B_{R}$,}\\
V_{n}^1&=0\qquad\text{on $\partial B_{R}$,}
\end{aligned}
\right.
\end{equation*}
and $V_{n}^2$ is harmonic in $B_{R}$. Since $U_{n}\le0$, $V_{n}^1$ is uniformly bounded in $B_{R}$.
Using the assumption \eqref{3061bis}, we also have
\begin{equation} \label{l1cont}
\int_{B(x_{n},r_{n}R)} v_{n}\;dx \le C (r_{n}R)^{N-2}.
\end{equation}
In other words, $V_{n}=-\Delta U_{n}$ is bounded in $L^1(B_{R})$. Since  $V_{n}^1$ is uniformly bounded in $B_{R}$, it follows that $V_{n}^2$ is bounded in $L^1(B_{R})$. Since $V_{n}^2$ is harmonic, $V_{n}^2$ is uniformly bounded in $B_{R/2}$. Similarly, write $U_{n}=U_{n}^1-U_{n}^2$, where
\begin{equation*}
\left\{
\begin{aligned}
-\Delta U_{n}^1&= V_{n}\qquad\text{in $B_{R/2}$,}\\
U_{n}^1&=0\qquad\text{on $\partial B_{R/2}$,}
\end{aligned}
\right.
\end{equation*}
and $U_{n}^2\ge0$ is harmonic in $B_{R/2}$. Since $V_{n}$ is bounded in $B_{R/2}$, so is $U_{n}^1$. Since $U_{n}^2(0)=U_{n}^1(0)$, we may now apply Harnack's inequality to conclude that $U_{n}^2$ is uniformly bounded in $B_{R/4}$. Hence, $U_{n}$ is uniformly bounded on compact subsets of $\R^N$.
So, we may pass to the limit in the first line of \eqref{fireg4} and find a solution $U$ of finite Morse index to
\eqref{gelfandbe}. Thanks to \eqref{l1cont}, if $V=-\Delta U$ and $R>0$, then
$$
\int_{B_{R}} V\;dx \le C R^{N-2},
$$
which is possible only if $\super V(\infty)=0$. This contradicts Theorem \ref{liouville}.

\hfill\qed

\appendix
\section{Appendix}
The following is a consequence of Kato's inequality. It was also proved by N. Ghoussoub et al. \cite{gh} in unpublished work. Its proof builds upon a similar result of P. Souplet \cite{souplet}, as well as similar inequalities in bounded domains obtained by C. Cowan, P. Esposito, and N. Ghoussoub \cite{ceg}.
\begin{proposition}\label{ogps}
Let  $u$ be a stable solution of \eq{gelfandbe} and $v=-\Delta u$. Then,
$$
v\geq \sqrt{2} e^{{\frac{u}{2}}}\quad\mbox{ in }\R^N.
$$
\end{proposition}
\begin{proof}
Let $w:=\sqrt{2}e^{{\frac{u}{2}}}-v$. A straightforward calculation yields
\neweq{ww}
\Delta w \geq \frac{1}{\sqrt 2}e^{{\frac{u}{2}}}w\quad\mbox{ in }\R^N.
\endeq
Since $v>0$ (see the proof of Lemma \ref{l1v}), we have $w^+\leq \sqrt{2}e^{{\frac{u}{2}}}$ in $\R^N$. We multiply \eq{ww} by $w^+$ and integrate over $B_R(0)$, for $R>0$. We obtain
\neweq{ww1}
\int\limits_{B_R(0)} |\nabla w^+|^2+\frac{1}{\sqrt 2}\int\limits_{B_R(0)}e^{{\frac{u}{2}}} |w^+|^2\leq \int\limits_{\partial B_R(0)} w^+\frac{\partial w^+}{\partial R}.
\endeq
For all $R>0$ define
$$
f(R):=\frac{1}{2}\int\limits_{\partial B_1(0)} |w^+(Rx)|^2d\sigma(x).
$$
Then \eq{ww1} reads
\neweq{ww2}
\int\limits_{B_R(0)} |\nabla w^+|^2+\frac{1}{\sqrt 2}\int\limits_{B_R(0)}e^{{\frac{u}{2}}} |w^+|^2\leq R^{N-1}f'(R).
\endeq
Since $u$ is stable and $e^u\geq 2|w^+|^2$ in $\R^N$, we have
$$
R^{N-4}\gtrsim \int\limits_{B_R(0)} e^u ds\geq \int_0^R r^{N-1}f(r)dr.
$$
In particular, $f$ cannot be strictly increasing in a given interval $[S,+\infty)$. Hence, there exists an increasing sequence $\{R_j\}$ such that
$R_j\rightarrow \infty$ and {$f'(R_j)\leq 0$.} Now, letting $R=R_j$ in \eq{ww2}, we find $w^+\equiv 0$, that is, $v\geq \sqrt{2} e^{{\frac{u}{2}}}$ in $\R^N$.

\end{proof}

\section*{Acknowledgement}This work has been supported by a Ulysses project between France and Ireland. The authors wish to thank the funding agencies Egide (France) and IRCSET (Ireland).

\bibliographystyle{amsalpha}







\end{document}